\documentclass[final,3p,times]{elsarticle}       % 单栏

\usepackage{lineno} %,hyperref
\modulolinenumbers[5]

\makeatletter
\def\ps@pprintTitle{%
  \let\@oddhead\@empty
  \let\@evenhead\@empty
  \let\@oddfoot\@empty
  \let\@evenfoot\@oddfoot
}
\makeatother
\usepackage[table,xcdraw,svgnames]{xcolor}
\usepackage[colorlinks]{hyperref}
\AtBeginDocument{%
  \hypersetup{
    citecolor=magenta,
    linkcolor=blue,
    urlcolor=Blue}}
\usepackage[heading=true,scheme=plain]{ctex} %%汉化英文模板 的时候   只想支持中文，标题不汉化
\usepackage{mathrsfs}
\usepackage{amsfonts}   % \mathbb{} 花体
\usepackage{amssymb}   %  \complement
\usepackage{lmodern}   % 数学公式字体
\usepackage{bbm}  % 小写字母空心 \mathbbm

\usepackage{autobreak}       %  autobreak package
\allowdisplaybreaks[4]   % 使数学公式自动分页

 \usepackage[capitalise]{cleveref}%\usepackage[nameinlink]{cleveref} %  宏包的\cref 类似于\cite引用多个文献

  \usepackage{extarrows}  % 定义箭头  \xlongequal  \xLongleftrightarrow
\usepackage{tikz}   % 绘图
\usetikzlibrary{arrows.meta}
\usetikzlibrary{snakes}
\usepackage{pgfplots}

\usepackage{subfigure}       % 插图
\numberwithin{figure}{section}
\usepackage[justification=centering]{caption}  % 标题换行
 \usepackage[shortlabels]{enumitem}  % 设定排列环境的 格式, 可以缩写简略形式
\numberwithin{equation}{section} %公式按章节编号
 \allowdisplaybreaks[4]   % 使数学公式自动分页
 \usepackage[amsmath,thmmarks]{ntheorem} % ntheorem.sty 是theorem.sty 的扩展,可以令定理等编号下端的黑点消失
\newtheorem{definition}{Definition} [section]            % 斜体
\newtheorem{theorem}{Theorem}[section]            % 斜体
\newtheorem{lemma}{Lemma} [section]%单独编号,与节有关
\newtheorem{proposition}{Proposition}[section]
\newdefinition{corollary}{Corollary}[section]
\newtheorem{remark}{Remark}
  \theoremsymbol{\hfill $\Box$} %{$\blacksquare$} {$\square$}
 \newtheorem*{proof}{Proof}
 \theoremstyle{empty}
 \newtheorem{refproof}{Proof}  % 定义 Proof of Theorem X  或者直接用宏包 \usepackage{amsthm} 替换 ntheorem
\def\loc{{\mathrm{loc}}}

\def\dchi{\scalebox{1.2}{$\chi$}}

\def\dint{\displaystyle\int}

\def\trace{{\rm{trace}}~}          % 迹 函数

\DeclareMathOperator*{\essinf}{ess\, inf}
\DeclareMathOperator*{\esssup}{ess\, sup}

\newcommand{\mathd}{\mathrm{d}}

\biboptions{numbers,sort&compress}   % 针对natbib，引用参考文献 缩写
\begin{document}

\begin{frontmatter}

\title{{\bfseries Boundness of Some   Multilinear Fractional Integral Operators in Generalized Morrey Spaces on Stratified Lie Groups }}

\author{Jianglong Wu\corref{mycorrespondingauthor}}
%\address{Department of Mathematics, Mudanjiang Normal University, Mudanjiang, 157011, China}
\cortext[mycorrespondingauthor]{Corresponding author}
\ead{jl-wu@163.com}


\author{Xiaojiao Tian}

\address{Department of Mathematics, Mudanjiang Normal University, Mudanjiang  157011, China}

\begin{abstract}
In this paper, the main aim  is to consider the  Spanne-type boundedness of  the multiliinear fractional integral operator $\mathcal{I}_{\alpha,m}$ and multiliinear fractional maximal operator $\mathcal{M}_{\alpha,m}$    in the generalized Morrey spaces over some stratified Lie group  $\mathbb{G}$.
\end{abstract}


\end{frontmatter}

%--------------
%\linenumbers    % 显示行数
%-------------

%==========================
\section{Introduction and main results}
\label{sec:introduction}

Stratified groups  appear in quantum physics and many parts of mathematics, including  several complex variables, Fourier analysis, geometry, and topology \cite{folland1982hardy,varopoulos2008analysis}.
The geometry structure of stratified groups is so good that it inherits a lot of analysis properties from the Euclidean spaces \cite{stein1993harmonic,grafakos2009modern}.
Apart from this, the difference between the geometry structures of Euclidean spaces and stratified groups makes the study of function spaces on them more complicated.
However, many harmonic analysis problems  on stratified Lie groups deserve a further investigation since  most results   of the theory of Fourier transforms and distributions in Euclidean spaces  cannot yet be duplicated \cite{liu2019multilinear,guliyev2022some}.

Nowadays, more and more attention has been paid to the study of function spaces which arise in the context of groups, such as variable
Lebesgue spaces\cite{liu2019multilinear,liu2022characterisation}, Orlicz spaces \cite{guliyev2022some} and generalized Morrey spaces \cite{guliyev2020commutators} et al.
And Morrey spaces were originally introduced by Morrey in \cite{morrey1938solutions} to study the local behavior of solutions to
second-order elliptic partial differential equations.

The multilinear fractional integral operators were first studied by Grafakos \cite{grafakos1992multilinear}, followed by Kenig and
Stein \cite{kenig1999multilinear} et al. The importance of fractional integral operators is due to the fact that they %are smooth operators and
have been widely used in various areas, such as potential analysis, harmonic analysis, and partial differential equations and so on\cite{wang2019multilinear}.

Let $\mathbb{G}$ be a stratified group. According to the definition of the classical multilinear fractional integral operator,   the multilinear
fractional integral operator  $\mathcal{I}_{\alpha,m}$ on stratified groups  can be defined by
\begin{align*}
\mathcal{I}_{\alpha,m}(\vec{f})(x)  & = \dint_{\mathbb{G}^{m}} \dfrac{f_{1}(y_{1}) \cdots f_{m}(y_{m})}{(\rho(y_{1}^{-1}x)+\cdots+\rho(y_{m}^{-1}x))^{mQ-\alpha}}  d \vec{y}, \ \ 0 \le \alpha < mQ,
\end{align*}
and the multilinear fractional maximal operator  $\mathcal{M}_{\alpha,m}$ is defined as follows
\begin{align*}
\mathcal{M}_{\alpha,m}(\vec{f})(x)  & = \sup_{B\ni x  \atop B \subset \mathbb{G}} |B|^{\frac{\alpha}{Q}} \prod_{i=1}^{m} \dfrac{1}{|B|} \dint_{B} |f_{i}(y_{i})|dy_{i}, \ \ 0 \le \alpha < mQ,
\end{align*}
where the supremum is taken over all $\mathbb{G}$-balls $B$ (see the notion in \cref{sec:preliminary}) containing $x$ with radius $r>0$, $|B|$ is the Haar measure of the   $\mathbb{G}$-ball $B$, and   vector function $\vec{f}=(f_{1},f_{2},\dots,f_{m})$ is locally integrable on $\mathbb{G}$.

If we take $m = 1$, % and $\mathbb{G}=\mathbb{R}^{n}$, classical
then $\mathcal{I}_{\alpha,m}$ is   a natural generalization of the classical fractional integral operator $I_{\alpha}   \equiv \mathcal{I}_{\alpha,1}$,  and the classical fractional maximal function $M_{\alpha}\equiv\mathcal{M}_{\alpha,1}  $ coincides for $\alpha= 0$ with the Hardy-Littlewood maximal function $M  \equiv \mathcal{M}_{0,1}$.

In 2013, Guliyev et al. \cite{guliyev2013boundedness} proved the boundedness of  the fractional maximal operator $\mathcal{M}_{\alpha,m} $  with $m=1$ in the generalized Morrey spaces $\mathcal{L}^{p,\varphi}(\mathbb{G})$ (see the definition below) on any Carnot group $\mathbb{G}$.
In 2017, Eroglu et al.  \cite{eroglu2017characterizations} studied the boundedness of the fractional integral operator $\mathcal{I}_{\alpha,m}$ with $m=1$ in the generalized Morrey spaces $\mathcal{L}^{p,\varphi}(\mathbb{G})$  on  Carnot group $\mathbb{G}$.
In 2014, Guliyev and Ismayilova \cite{guliyev2014multi} obtained the boundedness of  $\mathcal{M}_{\alpha,m} $ and $\mathcal{I}_{\alpha,m}$ on product generalized Morrey spaces with $\mathbb{G}=\mathbb{R}^{n}$.
And in recently, %In 2019,
Liu et al.  \cite{liu2019multilinear}  considered   the multilinear fractional integral $\mathcal{I}_{\alpha,m}$ in variable Lebesgue spaces on stratified groups.

Inspired by the above literature, the purpose of this paper is to   study  the  Spanne-type boundedness of  the  multiliinear fractional integral operator $\mathcal{I}_{\alpha,m}$ and multiliinear fractional maximal operator $\mathcal{M}_{\alpha,m}$  in the generalized Morrey spaces over some stratified Lie group  $\mathbb{G}$.
In all cases, the conditions for the boundedness of $\mathcal{I}_{\alpha,m}$ are given in terms of Zygmund-type integral inequalities on  $(\varphi_{1},\ldots,\varphi_{m},\psi)$
and the conditions for the boundedness of $\mathcal{M}_{\alpha,m}$  are given in terms of supremal type inequalities on $(\varphi_{1},\ldots,\varphi_{m},\psi)$,
which do not assume any assumption on monotonicity of $\varphi_{1},\ldots,\varphi_{m}$ and $\psi$ in $r$.

Our main result can be stated as follows.

The first result gives   the Spanne-type boundedness of multilinear fractional integral operator $\mathcal{I}_{\alpha,m}$ on product generalized Morrey space.

%-----------------
\begin{theorem}  \label{thm: frac-int-op-main-1} %[Spanne-type  result]
%---------------
  Suppose that $m\in \mathbb{Z}^{+}$,  $0<\alpha_{i} <Q $,  $1< p_{i}<Q/\alpha_{i}~(i=1,2,\ldots,m)$ and $\alpha=\sum\limits_{i=1}^{m}\alpha_{i}$.  Let  $q$ satisfy
%----------------
\begin{align*}
%-------
  \frac{1}{q}=\frac{1}{p_{1}}+\cdots+\frac{1}{p_{m}}-\frac{\alpha}{Q}<1,
%----------
\end{align*}
%------------
and  $(\varphi_{1},\ldots,\varphi_{m},\psi)$ satisfy
%----------------
\begin{align} \label{condition:Zygmund-type-integral-inequality}
%-------
  \prod_{i=1}^{m} \dint_{r}^{\infty} \dfrac{ \essinf\limits_{t<s<\infty}  \varphi_{i}(x,s) s^{Q/p_{i}}   }  {t^{Q/q_{i}}}  \dfrac{dt}{t} \leq C\psi(x,r),
%----------
\end{align}
%------------
where  $C>0$ does not depend on $r>0$ and $x\in \mathbb{G}$, and $ \frac{1}{q}=\sum\limits_{i=1}^{m}\frac{1}{q_{i}}$ with $ \frac{1}{q_{i}}=\frac{1}{p_{i}}-\frac{\alpha_{i}}{Q}~(i=1,2,\ldots,m)$. Then the operator $\mathcal{I}_{\alpha,m}$ is bounded from product space  $ \mathcal{L}^{p_{1},\varphi_{1}}(\mathbb{G})\times \cdots \times  \mathcal{L}^{p_{m},\varphi_{m}}(\mathbb{G})$ to $ \mathcal{L}^{q,\psi}(\mathbb{G})$.

%----------
\end{theorem}
%-----------------

The following result gives   the Spanne-type boundedness of multilinear fractional maximal  operator $\mathcal{M}_{\alpha,m}$ on product generalized Morrey space.
%-----------------
\begin{theorem} \label{thm: frac-max-op-main-2}  %[Spanne-type  result]
%---------------
  Suppose that $m\in \mathbb{Z}^{+}$,  $0<\alpha_{i} <Q $,  $1< p_{i}<Q/\alpha_{i}~(i=1,2,\ldots,m)$ and $\alpha=\sum\limits_{i=1}^{m}\alpha_{i}$.  Let  $q$ satisfy
%----------------
\begin{align*}
%-------
  \frac{1}{q}=\frac{1}{p_{1}}+\cdots+\frac{1}{p_{m}}-\frac{\alpha}{Q}<1,
%----------
\end{align*}
%------------
and  $(\varphi_{1},\ldots,\varphi_{m},\psi)$ satisfy
%----------------
\begin{align} \label{condition:supremal-type-integral-inequality}
%-------
  \prod_{i=1}^{m} \sup_{r<t<\infty} \dfrac{ \essinf\limits_{t<s<\infty}  \varphi_{i}(x,s) s^{Q/p_{i}}   }  {t^{Q/q_{i}}}  \leq C\psi(x,r),
%----------
\end{align}
%------------
where  $C>0$ does not depend on $r>0$ and $x\in \mathbb{G}$, and $ \frac{1}{q}=\sum\limits_{i=1}^{m}\frac{1}{q_{i}}$ with $ \frac{1}{q_{i}}=\frac{1}{p_{i}}-\frac{\alpha_{i}}{Q}~(i=1,2,\ldots,m)$. Then the operator $\mathcal{M}_{\alpha,m}$ is bounded from product space  $ \mathcal{L}^{p_{1},\varphi_{1}}(\mathbb{G})\times \cdots \times  \mathcal{L}^{p_{m},\varphi_{m}}(\mathbb{G})$ to $ \mathcal{L}^{q,\psi}(\mathbb{G})$.

%----------
\end{theorem}
%-----------------

%--------------------
\begin{remark}  \label{rem.main-result}
%-------------------------------
\begin{enumerate}[ label=(\roman*)]  %,itemindent=-0.3em,itemindent=1.5em
%--------------------------------------------------------------
\item When $\mathbb{G}=\mathbb{R}^{n}$, the conclusion in \cref{thm: frac-int-op-main-1} can be found from \cite{guliyev2014multi} (see Theorem 4.3).
%-----------------
\item In \cref{thm: frac-int-op-main-1},  when $m=1$,  the result   has been considered in \cite{eroglu2017characterizations}(see Theorem 4.3).
%-------
\item  In  \cref{thm: frac-max-op-main-2}, when $m=1$,     the conclusion  can be found in  \cite{guliyev2013boundedness}(see Theorem 3.2).
% Furthermore, in the case  $\alpha = 0$ and $p = q$, the conclusion coincides with the Corollary 3.1 in \cite{guliyev2013boundedness}.
%-------
\item In the case $\alpha = 0$  from  \cref{thm: frac-max-op-main-2}, the  above result is also true. And when  $m=1$, the conclusion coincides with the Corollary 3.1 in \cite{guliyev2013boundedness}.
%------------
\end{enumerate}
%--------------------------
\end{remark}
%-------------

 Throughout this paper, the letter $C$  always stands for a constant  independent of the main parameters involved and whose value may differ from line to line.
In addition, we  give some notations. Here and hereafter $|E|$  will always denote the Haar measure of a measurable set $E$ of $\mathbb{G}$ and by  \raisebox{2pt}{$\dchi_{E}$} denotes the  characteristic function of a measurable set $E \subset\mathbb{G}$.
Let $L^{p} ~(1\le p\le \infty)$  be the standard $L^{p} $-space with respect to the Haar measure $\mathd x$.
For a measurable set $E \subset\mathbb{G}$ and a positive integer $m$, we will use the notation $(E)^{m}=\underbrace{E\times \cdots \times E}_{m}$ sometimes. And we will occasionally use the notational $\vec{f}=(f_{1},\dots , f_{m})$, $T(\vec{f})=T(f_{1},\dots , f_{m})$, $\mathd\vec{y}=dy_{1}\cdots  dy_{m}$ and $(x,\vec{y})=(x,y_{1},\dots , y_{m})$ for convenience.

%This paper is organized as follows. In the next section, we recall some basic definitions and known results. In  \cref{sec:proof-mab}, we will prove  \cref{thm:lipschitz-frac-main-1}.   \cref{sec:proof-nonlinear}  is devoted to proving \cref{thm:lipschitz-nonlinear-frac-main-1}.

\section{Preliminaries and lemmas}
\label{sec:preliminary}

To prove the main results of this paper, we first recall some necessary notions and remarks.
Firstly, we recall some preliminaries concerning stratified Lie groups (or so-called
Carnot groups). We refer the reader to  \cite{folland1982hardy,bonfiglioli2007stratified,stein1993harmonic}.

\subsection{Lie group $\mathbb{G}$}

%------------------------
 \begin{definition}%[\citet{krantz1982lipschitz}]
 \label{def:stratified-Lie-algebra-krantz1982lipschitz}
%----------------------------
  Let $m\in \mathbb{Z}^{+}$, $\mathcal{G}$  be a finite-dimensional Lie algebra,  $[X, Y] = XY - YX \in \mathcal{G}$ be Lie bracket with $X,Y  \in \mathcal{G}$.
%-------------------------------
\begin{enumerate}[label=(\arabic*)]  %leftmargin=2em,,,itemindent=1.0em
%--------------------------------------------------------------
\item If $Z \in \mathcal{G}$ is an $m^{\text{th}}$  order Lie bracket and $W \in \mathcal{G}$, then $[Z,W]$ is an $(m + 1)^{\text{st}}$ order  Lie bracket.
%----------------
\item  We say $\mathcal{G}$  is a nilpotent Lie algebra of step  $m$ if  $m$ is the smallest integer for which all Lie brackets of order $m+1$ are zero.
%-------
\item   We say that  a   Lie algebra $\mathcal{G}$   is  stratified  if there is a direct sum vector space decomposition
%-----------------
\begin{align}\label{equ:lie-algebra-decomposition}
%-------
 \mathcal{G} =\oplus_{j=1}^{m} V_{j}  = V_{1} \oplus  \cdots \oplus  V_{m}
%----------
\end{align}
%----------
such that $\mathcal{G}$ is nilpotent of step $m$, that is,
%-----------------
\begin{align*}
%-------
 [V_{1},V_{j}] =
%------------
 \begin{cases}
 %----------
 V_{j+1} &  1\le j \le  m-1  \\
 0 & j\ge m
%------------
\end{cases}
%----------
\end{align*}
%------------
 holds.
%------------
\end{enumerate}
%---------------
%----------------
\end{definition}
%----------

It is not difficult to find that the above $V_{1}$  generates the whole of the Lie algebra $\mathcal{G}$ by taking Lie brackets since each element of $ V_{j}~(2\le j \le m)$  is a linear combination of $(j-1)^{\text{th}}$ order Lie bracket  of elements of $ V_{1}$.

With the help of the related  notions of Lie algebra (see \cref{def:stratified-Lie-algebra-krantz1982lipschitz}), the following definition can be obtained.
%------------------------
 \begin{definition}\label{def:stratified-Lie-group}
%----------------------------
  Let $\mathbb{G}$ be a finite-dimensional, connected and simply-connected Lie group   associated with Lie algebra $\mathcal{G}$. Then
%-------------------------------
\begin{enumerate}[label=(\arabic*)]  %leftmargin=2em,,,itemindent=1.0em
%--------------------------------------------------------------
\item  $\mathbb{G}$  is called nilpotent if its Lie algebra $\mathcal{G}$ is nilpotent.
%----------------
\item  $\mathbb{G}$ is said to be stratified if its Lie algebra $\mathcal{G}$ is stratified.
%-------
\item  $\mathbb{G}$ is called homogeneous if it is a nilpotent Lie group whose Lie algebra $\mathcal{G}$ admits a family of dilations $\{\delta_{r}\}$, namely, for $r>0$, $X_{k}\in V_{k}~(k=1,\ldots,m)$,
%-----------------
\begin{align*}
%-------
  \delta_{r} \Big( \sum_{k=1}^{m} X_{k} \Big)  =  \sum_{k=1}^{m} r^{k} X_{k},
%----------
\end{align*}
%------------
which are Lie algebra automorphisms.
%------------
\end{enumerate}
%---------------
%----------------
\end{definition}
%----------

%--------------------------------
\begin{remark}  \label{rem:lie-algebra-decom-zhu2003herz}  %
%--------------------------------
Let $\mathcal{G} =  \mathcal{G}_{1}\supset  \mathcal{G}_{2} \supset \cdots \supset  \mathcal{G}_{m+1} =\{0\}$   denote the lower central series of  $\mathcal{G}$, and $X=\{X_{1},\dots,X_{n}\}$ be a basis for $V_{1}$ of $\mathcal{G}$.
%-------------------------------
\begin{enumerate}[label=(\roman*) ]  %leftmargin=2em,,itemindent=-0.3em
%--------------------------------------------------------------
\item  (see \cite{zhu2003herz})  The direct sum decomposition  \labelcref{equ:lie-algebra-decomposition} can be constructed by identifying each $\mathcal{G}_{j}$ as a vector subspace of $\mathcal{G}$ and setting $ V_{m}=\mathcal{G}_{m}$ and $ V_{j}=\mathcal{G}_{j}\setminus \mathcal{G}_{j+1}$ for $j=1,\ldots,m-1$.
%------------
\item   (see \cite{folland1979lipschitz}) The number $Q=\trace A =\sum\limits_{j=1}^{m} j\dim(V_{j})$ is called the homogeneous dimension of $\mathcal{G}$, where $A$ is a diagonalizable linear transformation of  $\mathcal{G}$ with positive eigenvalues.
%------------
\item  (see \cite{zhu2003herz} or \cite{folland1979lipschitz})   %The  homogeneous dimension of  $\mathbb{G}$ at infinity as the integer $Q$ is given by
The number  $Q$ is also called the homogeneous dimension of $\mathbb{G}$  since $\mathd(\delta_{r}x)=r^{Q}\mathd x$ for all $r>0$, and
%-----------------
\begin{align*}
%-------
 Q = \sum_{j=1}^{m} j \dim(V_{j}) = \sum_{j=1}^{m} \dim(\mathcal{G}_{j}).
%----------
\end{align*}
%------------
\end{enumerate}
%-----------------------
\end{remark}
%------------------------

By the Baker-Campbell-Hausdorff formula for sufficiently small elements $X$ and $Y$ of $\mathcal{G}$ one has
%------------------
\begin{align*}
%-------
 \exp (X) \exp (Y)=  \exp (H(X,Y))= X+Y +\frac{1}{2}[X,Y]+\cdots
%----------
\end{align*}
%------------
where $\exp : \mathcal{G} \to \mathbb{G}$ is the exponential map, $H(X, Y )$ is an infinite linear
combination of $X$ and $Y$ and their Lie brackets, and the dots denote terms of order higher than two. And the above equation is finite in the case  of  $\mathcal{G}$ is a nilpotent Lie algebra.

The following properties can be found in \cite{ruzhansky2019hardy}(see Proposition 1.1.1, or  Proposition 2.1 in \cite{yessirkegenov2019function} or Proposition 1.2 in \cite{folland1982hardy}).

%--------------------------------
\begin{proposition}\label{pro:2.1-yessirkegenov2019}
%------------------------------
 Let $\mathcal{G}$ be a nilpotent Lie algebra, and let $\mathbb{G}$ be the corresponding connected and simply-connected nilpotent Lie group. Then we have
%------------------------------
\begin{enumerate}[leftmargin=2em,label=(\arabic*),itemindent=1.0em]  %,itemindent=-0.3em
%--------------------------------
\item   The exponential map  $\exp: \mathcal{G} \to \mathbb{G}$  is a diffeomorphism. Furthermore, the group law $(x,y) \mapsto xy$ is a polynomial map if  $\mathbb{G}$ is identified with $\mathcal{G}$ via $\exp$.
%-------------------------------------
\item  If $\lambda$ is a Lebesgue measure on  $\mathcal{G}$, then $\exp\lambda$ is a bi-invariant Haar measure on  $\mathbb{G}$ (or a bi-invariant Haar measure $\mathd  x$ on  $\mathbb{G}$  is just the lift of Lebesgue measure on  $\mathcal{G}$ via $\exp$).
%-------------------------------
\end{enumerate}
%-----------------------------
\end{proposition}
%------------------------

Thereafter, we use $Q$ to denote the homogeneous dimension of  $\mathbb{G}$, $y^{-1}$ represents the inverse of $y\in \mathbb{G}$,  $y^{-1}x$ stands for the group multiplication of $y^{-1}$  by $x$ and the group identity  element of $\mathbb{G}$ will be referred to as the origin denotes by $e$.

A homogeneous norm on $\mathbb{G}$ is a continuous function $x\to \rho(x)$  from $\mathbb{G}$ to $[0,\infty)$, which is  $C^{\infty}$ on $\mathbb{G}\setminus\{0\}$ and satisfies
%-----------------
\begin{align*}
%-------
\begin{cases}
%------------
 \rho(x^{-1}) =  \rho(x), \\
 \rho(\delta_{t}x) =  t\rho(x) \ \ \text{for all}~  x \in \mathbb{G} ~\text{and}~ t > 0, \\
%------------
 \rho(e) =  0.
%------------
\end{cases}
%----------
\end{align*}
%------------
Moreover, there exists a constant $c_{0} \ge 1$ such that $\rho(xy) \le c_{0}(\rho(x) + \rho(y))$ for all $x,y \in \mathbb{G}$.

 With the norm above, we define the $\mathbb{G}$ ball centered at $x$ with radius $r$ by $B(x, r) = \{y \in \mathbb{G} : \rho(y^{-1}x) < r\}$,  and by $\lambda B$ denote the ball $B(x,\lambda r)$  with $\lambda>0$, let $B_{r} = B(e, r) = \{y \in \mathbb{G}  : \rho(y) < r\}$ be the open ball centered at $e$ with radius $r$,  which is the image under $\delta_{r}$ of $B(e, 1)$.
 And by $\sideset{^{\complement}}{}  {\mathop {B(x,r)}} = \mathbb{G}\setminus B(x,r)= \{y \in \mathbb{G} : \rho(y^{-1}x) \ge r\}$ denote the complement of $B(x, r)$.  Let  $|B(x,r)|$ be the Haar measure of the ball  $B(x,r)\subset \mathbb{G}$, and
 there exists $c_{1} =c_{1} (\mathbb{G})$ such that
%----------------
\begin{align*}
%-------
  |B(x,r)| = c_{1} r^{Q}, \ \  \ \   x\in \mathbb{G}, r>0.
%----------
\end{align*}
%-----------
In addition,  the Haar measure of a homogeneous Lie group  $\mathbb{G}$  satisfies the doubling condition (see pages 140 and 501,\cite{fischer2016quantization}), i.e. $\forall ~ x\in \mathbb{G}$, $r>0$, $\exists~ C$, such that
%----------------
\begin{align*}
%-------
   |B(x,2r)| \le C |B(x,r)|.
%----------
\end{align*}
%-----------

The most basic partial differential operator in a stratified Lie group is the sub-Laplacian associated with $X=\{X_{1},\dots,X_{n}\}$, i.e., the second-order partial differential operator on
 $\mathbb{G}$  given by
%--------------------------------
\begin{align*}
%-------
 \mathfrak{L} =  \sum_{i=1}^{n} X_{i}^{2}.
%----------
\end{align*}
%-----------

The part \labelcref{enumerate:holder-Lie} in following lemma is known as the  H\"{o}lder's inequality on Lebesgue spaces over Lie groups $\mathbb{G}$, it can also be found in \cite{rao1991theory} or \cite{guliyev2022some}, when  Young function $\Phi(t)=t^{p}$ and its complementary function $\Psi(t)=t^{q}$ with $\frac{1}{p}+\frac{1}{q}=1$.  And by simple calculations, the part \labelcref{enumerate:multi-holder-Lie} can be  deduced  from the part \labelcref{enumerate:holder-Lie}.  % straightforward to show
%--------------------------------
\begin{lemma}[H\"{o}lder's inequality on  $\mathbb{G}$]\label{lem:holder-inequality-Lie-group}
Let   $\Omega\subset  \mathbb{G}$ be a measurable set.
%--------------------------
\begin{enumerate}[fullwidth,label=(\arabic*)]
%------------
\item Suppose that  $1\le p,q \le\infty$ with $\frac{1}{p}+\frac{1}{q}=1$,   and measurable functions $f\in L^{p}(\Omega)$ and $g\in L^{q}(\Omega)$.  Then there exists a positive constant $C$ such that
%--------------------------------
\begin{align*} %\label{inequ:holder-inequality-Lie-group}
%-------
   \dint_{\Omega} |f(x)g(x)|  \mathrm{d}x \le C \|f\|_{L^{p}(\Omega)} \|g\|_{L^{q}(\Omega)}.
%----------
\end{align*}
%------------
%-----------
    \label{enumerate:holder-Lie}
%----------------
\item  Suppose that     $1< q_{i}<\infty~(i=1,2,\ldots,m)$ and   $q$ satisfy $\frac{1}{q}=\frac{1}{q_{1}}+\cdots+\frac{1}{q_{m}}$.
Then there exists a positive constant $C$ such that the inequality
%----------------
\begin{align*}
%-------
  \|f_{1}\cdots f_{m}\|_{L^{q}(\Omega)}\le C\prod_{i=1}^{m} \|f_{i}\|_{L^{q_{i}}(\Omega)}.
%----------
\end{align*}
%------------
holds for all  $f_{i} \in L^{q_{i}}(\Omega) ~(i=1,2,\ldots,m)$.
%-----------
    \label{enumerate:multi-holder-Lie}
%----------
\end{enumerate}
%--------------
\end{lemma}
%-------------

By elementary calculations we have the following property. It can also  be found in \cite{guliyev2022some}, when  Young function $\Phi(t)=t^{p}$.
%--------------------------------
\begin{lemma}[Norms of characteristic functions]\label{lem:norm-characteristic-functions-Lie-group}
%--------------------------
Let $0<p<\infty$ and $\Omega\subset  \mathbb{G}$ be a measurable set with finite Haar measure. Then
%--------------------------------
\begin{align*} %\label{inequ:holder-inequality-Lie-group}
%-------
  \|\dchi_{\Omega}\|_{L^{p}(\mathbb{G})} = \|\dchi_{\Omega}\|_{WL^{p}(\mathbb{G})}  = |\Omega|^{1/p}.
%----------
\end{align*}
%------------
%------------------------------------------
\end{lemma}
%-------------

%------------------------
\subsection{Morrey spaces  on $\mathbb{G}$}
%------------------------

%Morrey spaces were originally introduced by Morrey in \cite{morrey1938solutions} to study the local behavior of solutions to second-order elliptic partial differential equations.
Morrey spaces, named after Morrey, seem to describe the boundedness property of the classical fractional integral operators more precisely than Lebesgue spaces \cite{iida2012multilinear}.

%--------------------------------
\begin{definition}[Morrey-type spaces on $\mathbb{G}$]   \label{def.morrey-space} \ Let  $1\le p <\infty$, and $B=B(x,r)$ be a $\mathbb{G}$-ball  centered at $x$ with radius $r>0$.
%------------------------------
\begin{enumerate}[ label=(\arabic*)]
%-------------------------------------
\item %(see \citet{eroglu2017characterizations})
When $0\le \lambda \le Q$. The Morrey-type space $ L^{p,\lambda}(\mathbb{G})$ is defined by
%----------------
\begin{align*}
%-------
  L^{p,\lambda}(\mathbb{G})  &= \{ f\in L_{\loc}^{p}(\mathbb{G}): \|f\|_{L^{p,\lambda}(\mathbb{G})} < \infty \}
%----------
\end{align*}
%------------
with
%----------------
\begin{align*}
%-------
 \|f\|_{L^{p,\lambda}(\mathbb{G})}  &= \sup_{x \in \mathbb{G} \atop r>0 } \Big( \frac{1}{\lvert B\rvert ^{\lambda/Q}} \dint_{B} \vert f(y)\vert^{p} \mathd y \Big)^{1/p}.
%----------
\end{align*}
%------------
%where the supremum is taken over every ball $B\subset \mathbb{G}$ containing $x$ with radius $r>0$.
%-----------
    \label{enumerate:def-morrey-1}
%-------------------------------------
\item  %(see \citet{eroglu2017characterizations})
Set $\varphi(x,r)$ be a positive measurable function on $\mathbb{G}\times (0,\infty)$.   The generalized Morrey space $ \mathcal{L}^{p,\varphi}(\mathbb{G})$ is defined for all functions $f\in L_{\loc}^{p}(\mathbb{G})$ with the finite quasinorm
%----------------
\begin{align*}
%-------
     \|f\|_{\mathcal{L}^{p,\varphi}(\mathbb{G})}  &= \sup_{ x \in\mathbb{G} \atop r>0} \dfrac{1}{\varphi(x,r)} \Big( \frac{1}{\vert B\vert } \dint_{B} \vert f(y)\vert ^{p} \mathd y \Big)^{1/p}.
%----------
\end{align*}
%------------
%%----------------
%\begin{align*}
%%-------
%     \|f\|_{\mathcal{L}^{p,\varphi}(\mathbb{G})}  &= \sup_{B\ni x \atop B\subset \mathbb{G}} \dfrac{1}{\varphi(x,r)} \Big( \frac{1}{\vert B\vert } \dint_{B} \vert f(y)\vert ^{p} \mathd y \Big)^{1/p},
%%----------
%\end{align*}
%%------------
%where the supremum is taken over every ball $B\subset \mathbb{G}$ containing $x$ with radius $r>0$.
%-----------
    \label{enumerate:def-morrey-2}
%-----------------
\item  Furthermore, the weak generalized Morrey space $ W\mathcal{L}^{p,\varphi}(\mathbb{G})$ is defined for all functions $f\in L_{\loc}^{p}(\mathbb{G})$ by the
finite norm
%----------------
\begin{align*}
%-------
     \|f\|_{W\mathcal{L}^{p,\varphi}(\mathbb{G})}  &= \sup_{x \in\mathbb{G} \atop r>0} \dfrac{|B|^{-\frac{1}{p}}}{\varphi(x,r)}    \|f\|_{WL^{p}(B)}.
%----------
\end{align*}
%------------
%%----------------
%\begin{align*}
%%-------
%     \|f\|_{W\mathcal{L}^{p,\varphi}(\mathbb{G})}  &= \sup_{B\ni x \atop B\subset \mathbb{G}} \dfrac{|B|^{-\frac{1}{p}}}{\varphi(x,r)}    \|f\|_{WL^{p}(B)},
%%----------
%\end{align*}
%%------------
%where the supremum is taken over every ball $B\subset \mathbb{G}$  containing $x$ with radius $r>0$.
%-----------
    \label{enumerate:def-weak-morrey-3}
%-------------------------------------
%----------
\end{enumerate}
%-------------------
\end{definition}
%--------------------

%--------------------
\begin{remark}[see \cite{eroglu2017characterizations} or \cite{guliyev2020characterizations}]  \label{rem.morrrey-def}
%-------------------------------
\begin{enumerate}[ label=(\roman*)]  %,itemindent=-0.3em,itemindent=1.5em
%--------------------------------------------------------------
\item  It is well known that if   $1\le p <\infty$ then
%-----------------
\begin{align*}
%-------
  L^{p,\lambda}(\mathbb{G})  =
%-------
\begin{cases}
%------------
   L^{p}(\mathbb{G})  & \text{if}\ \lambda=0, \\
  L^{\infty}(\mathbb{G})  & \text{if}\ \lambda=Q,\\
  \Theta             & \text{if}\ \lambda<0 \ \text{or}\ \lambda>Q,
%------------
\end{cases}
%----------
\end{align*}
%------------
where $\Theta$ is the set of all functions equivalent to $0$ on $\mathbb{G}$.
%-------------------------------------
\item  In \labelcref{enumerate:def-morrey-2},  when   $1\le p<\infty $ %$1\le p\le q $, $\kappa \in (0,1)$
and  $0\le \lambda \le Q$, we have  $\mathcal{L}^{p,\varphi}(\mathbb{G})  = L^{p,\lambda}(\mathbb{G}) $
%-----------------
if $\varphi(x,r)=\vert B\vert ^{(\lambda/Q-1)/p}$ and $B\subset \mathbb{G}$ denotes the ball  with radius   $r$ and  containing $x$.
%------------
\end{enumerate}
%--------------------------
\end{remark}
%----------------------

Now, we give some necessary notations and notions. Let  $w$ be a weight function and $u$ be a continuous and non-negative function on $(0,\infty)$.
%---------------------------
\begin{itemize}[itemsep= -6 pt]  %,topsep = 1em
%------------
 \item   We denote by $L_{w}^{p}(\Omega)~(1<p<\infty)$ the weight $L^{p}(\Omega)$ space of all functions $f$ measurable on  a measurable set $\Omega\subset \mathbb{G}$ with $\|f\|_{L_{w}^{p}(\Omega)}= \|wf\|_{L^{p}(\Omega)}<\infty$.
%------------
 \item By $L_{w}^{\infty}(0,\infty)$ denotes the weight $L^{\infty}(0,\infty)$ space of all functions  $g(t)$ measurable on $(0,\infty)$  with finite norm
%----------------
\begin{align*}
%-------
   \|g\|_{L_{w}^{\infty}(0,\infty)} = \esssup_{t>0}    w(t) |g(t)|.
%----------
\end{align*}
%------------------
 \item $L_{1}^{\infty}(0,\infty) \xlongequal[]{w=1}L^{\infty}(0,\infty)$.
%------------------
 \item Let $\mathfrak{M}(0,\infty)$ be the set of all Lebesgue measurable functions on $(0,\infty)$ and  $\mathfrak{M}^{+}(0,\infty)$ its subset consisting of all non-negative functions on $(0,\infty)$.
%------------------
 \item Denote by $\mathfrak{M}^{+}(0,\infty;\uparrow)$ the cone of all functions in $\mathfrak{M}^{+}(0,\infty)$ which  are non-decreasing on $(0,\infty)$ and set
%----------------
\begin{align*}
%-------
  \mathbb{A}= \{ \varphi\in \mathfrak{M}^{+}(0,\infty;\uparrow): \lim_{t\to 0^{+}} \varphi(t) =0\}.
%----------
\end{align*}
%--------------
 \item We define the supremal operator $\overline{S}_{u}$ on $g\in\mathfrak{M}(0,\infty)$  by
%----------------
\begin{align*}
%-------
  ( \overline{S}_{u}g)(t) =  \|ug\|_{L^{\infty}(t,\infty)}, \qquad t\in (0,\infty).
%----------
\end{align*}
%------------
\end{itemize}
\section{Proofs of the main results} %Main results and their proofs} %  the principal results
\label{sec:result-proof}

Now we give the proofs of the \cref{thm: frac-int-op-main-1} and \cref{thm: frac-max-op-main-2}. %main results.

\subsection{Proof of \cref{thm: frac-int-op-main-1}}

In order to prove \cref{thm: frac-int-op-main-1}, we also need some auxiliary results.
 The following lemma can be obtained from \cite{liu2019multilinear} with $p_{i}(x)~(i=1,2,\ldots,m)$ is constant (see P.111--114). And in the case $m=1$, the following result can be founded in \cite{guliyev2010stein} (see theorem 2.5 or \cite{folland1982hardy}).
%----------------
\begin{lemma}    \label{lem:thm-3.1-liu2019multilinear}
%---------------
Suppose that  $m\in \mathbb{Z}^{+}$,  $0<\alpha_{i} <Q $,  $1< p_{i}<Q/\alpha_{i}~(i=1,2,\ldots,m)$ and $\alpha=\sum\limits_{i=1}^{m}\alpha_{i}$. Let  $q$ satisfy
%----------------
\begin{align*}
%-------
  \frac{1}{q}=\frac{1}{p_{1}}+\cdots+\frac{1}{p_{m}}-\frac{\alpha}{Q}<1.
%----------
\end{align*}
%------------
Then the operator $\mathcal{I}_{\alpha,m}$ is bounded from product space  $ L^{p_{1}}(\mathbb{G})\times \cdots \times  L^{p_{m}}(\mathbb{G})$ to $ L^{q}(\mathbb{G})$, namely
%----------------
\begin{align*}
%-------
  \|\mathcal{I}_{\alpha,m}(\vec{f})\|_{L^{q}(\mathbb{G})}\le C\prod_{i=1}^{m} \|f_{i}\|_{L^{p_{i}}(\mathbb{G})}.
%----------
\end{align*}
%------------
%----------
\end{lemma}
%------------

%%Zygmund-type integral inequality
The following pointwise estimate is also necessary, and it is proved in \cite{guliyev2013generalized} (or \cite{eroglu2017characterizations}).
%----------------
\begin{lemma}    \label{lem:thm-3.1-guliyev2013generalized}
%---------------
Let $v_{1}$, $v_{2}$ and $w$ be positive   weight functions on $(0,\infty)$, and let $v_{1}$ be bounded outside a neighborhood of the origin.  Then the   inequality
%----------------
\begin{align} \label{inequ:3.1-guliyev2013generalized}
%-------
  \esssup _{t>0} v_{2}(t) H_{\omega}g(t) \leq C  \esssup _{t>0} v_{1}(t)  g(t)
%----------
\end{align}
%------------
holds for some $C>0$ and   all nonnegative and nondecreasing function $g$ on  $(0,\infty)$ if and only if
%----------------
\begin{align*}
%-------
 B:= \sup_{t>0} v_{2}(t) \dint_{t}^{\infty} \dfrac{ \omega(s) \mathd s }  {\esssup\limits_{s<\tau<\infty} v_{1}(\tau)} <\infty.
%----------
\end{align*}
%------------
where
%----------------
\begin{align*}
%-------
 H_{\omega}g(t) = \dint_{t}^{\infty} g(s) \omega(s) \mathd s, \ \ 0<t<\infty.
%----------
\end{align*}
%------------
Moreover, the value $C = B$ is the best constant for \labelcref{inequ:3.1-guliyev2013generalized}.
%----------
\end{lemma}
%------------

In addition, the following local estimates are valid.  %is also needed.
%----------------
\begin{lemma}    \label{lem:frac-loc-lebesgue}
%---------------
Suppose that  $m\in \mathbb{Z}^{+}$,  $0<\alpha_{i} <Q $,  $1< p_{i}<Q/\alpha_{i}~(i=1,2,\ldots,m)$ and $\alpha=\sum\limits_{i=1}^{m}\alpha_{i}$. Let  $q$ satisfy
%----------------
\begin{align*}
%-------
  \frac{1}{q}=\frac{1}{p_{1}}+\cdots+\frac{1}{p_{m}}-\frac{\alpha}{Q}<1.
%----------
\end{align*}
%------------
Then  the inequality
%----------------
\begin{align*}
%-------
  \|\mathcal{I}_{\alpha,m}(\vec{f})\|_{L^{q}(B(x,r))}\le C r^{Q/q}\prod_{i=1}^{m} \dint_{2r}^{\infty} t^{\alpha_{i}-\frac{Q}{p_{i}}-1} \|f_{i}\|_{L^{p_{i}}(B(x,t))}\mathd t
%----------
\end{align*}
%------------
holds for any ball $B(x,r)$ and for all  $\vec{f}\in L_{\loc}^{p_{1}}(\mathbb{G})\times \cdots \times  L_{\loc}^{p_{m}}(\mathbb{G})$.
%----------
\end{lemma}
%------------

%--------------------
\begin{proof}
%---------------
For arbitrary $x\in\mathbb{G}$, set $B=B(x,r)$ be the $\mathbb{G}$ ball centered at $x$ with radius $r$, and $2B=B(x,2r)$. For each $j$, we decompose $f_{j}=f_{j}^{0}+f_{j}^{\infty}$ with $f_{j}^{0}=f_{j}\dchi_{2B}$.
Then
\begin{align*}
 \prod_{j=1}^{m} f_{j}   &=  \prod_{j=1}^{m} \Big( f_{j}^{0} + f_{j}^{\infty} \Big)
   =   \sum_{\beta_{1},\dots,\beta_{m}\in \{0,\infty\}}  f_{1}^{\beta_{1}} \cdots  f_{m}^{\beta_{m}}  \\
  &=  \prod_{j=1}^{m} f_{j}^{0}  +  \sum_{(\beta_{1},\dots,\beta_{m})\in \ell} f_{1}^{\beta_{1}}  \cdots  f_{m}^{\beta_{m}},
\end{align*}
where $\ell=\{ (\beta_{1},\dots,\beta_{m}): \hbox{there is at least one}~ \beta_{j}\neq 0\}$.
Thus, for arbitrary $y\in B(x,r)$, we  obtain
\begin{align*}
\mathcal{I}_{\alpha,m}(\vec{f})(y) = \mathcal{I}_{\alpha,m} ( f_{1}^{0} ,\ldots, f_{m}^{0} )(y) +  \sum_{(\beta_{1},\dots,\beta_{m})\in \ell} \mathcal{I}_{\alpha,m} ( f_{1}^{\beta_{1}} ,\ldots, f_{m}^{\beta_{m}} ) (y).
\end{align*}
Then,
\begin{align*}
\|\mathcal{I}_{\alpha,m}(\vec{f})\|_{L^{q}(B(x,r))} &\leq \|\mathcal{I}_{\alpha,m} ( f_{1}^{0} ,\ldots, f_{m}^{0} )\|_{L^{q}(B(x,r))}
  + \bigg\|\sum_{(\beta_{1},\dots,\beta_{m})\in \ell} \mathcal{I}_{\alpha,m} ( f_{1}^{\beta_{1}} ,\ldots, f_{m}^{\beta_{m}} ) \bigg\|_{L^{q}(B(x,r))}   \\
&= E_{1}+E_{2}.
\end{align*}

For $E_{1}$, applying  the boundedness of  $\mathcal{I}_{\alpha,m}$ (see \cref{lem:thm-3.1-liu2019multilinear}), we have
%----------------
\begin{align*}
%-------
  E_{1} &= \|\mathcal{I}_{\alpha,m} ( f_{1}^{0} ,\ldots, f_{m}^{0} )\|_{L^{q}(B(x,r))} \le \|\mathcal{I}_{\alpha,m}( f_{1}^{0} ,\ldots, f_{m}^{0} )\|_{L^{q}(\mathbb{G})} \\
  &\le C\prod_{i=1}^{m} \|f_{i}^{0}\|_{L^{p_{i}}(\mathbb{G})}\le C\prod_{i=1}^{m} \|f_{i}\|_{L^{p_{i}}(B(x,2r))}.
%----------
\end{align*}
%------------

%--------------------
Applying the continuous version of Minkowski's inequality and   doubling condition of Haar measure, it can obtain the following fact
\begin{align}\label{inequ:norm-single-integral}
\|f_{i}\|_{L^{p_{i}}(B(x,2r))} \le C r^{\frac{Q}{q_{i}}} \int_{2r}^{\infty} t^{\alpha_{i}-\frac{Q}{p_{i}}-1} \|f_{i}\|_{L^{p_{i}}(B(x,t))} \mathd t, \qquad i=1,2,\ldots,m.
\end{align}

Therefore, we get
%----------------
\begin{align*}
%-------
  E_{1} &\le  C\prod_{i=1}^{m}   r^{\frac{Q}{q_{i}}} \int_{2r}^{\infty} t^{\alpha_{i}-\frac{Q}{p_{i}}-1} \|f_{i}\|_{L^{p_{i}}(B(x,t))} \mathd t   \\
  &\le  C  r^{\frac{Q}{q}} \prod_{i=1}^{m}   \int_{2r}^{\infty} t^{\alpha_{i}-\frac{Q}{p_{i}}-1} \|f_{i}\|_{L^{p_{i}}(B(x,t))} \mathd t.
%----------
\end{align*}
%------------

To estimate $E_{2}$, we consider first the case   $\beta_{1}=\beta_{2}=\cdot\cdot\cdot=\beta_{m}=\infty$. When $y\in B(x,r)$ and $z_{i}\in  \sideset{^{\complement}}{}{\mathop {B(x,2r)}}=\mathbb{G}\setminus B(x,2r)~(i=1,2,\ldots,m)$,   the conditions imply $\frac{1}{2}\rho(z_{i}^{-1}y) \leq \rho(z_{i}^{-1}x) \leq \frac{3}{2}\rho(z_{i}^{-1}y)$. Thus, we can obtain
%----------------
\begin{align*}
%-------
 | \mathcal{I}_{\alpha,m} ( f_{1}^{\infty} ,\ldots, f_{m}^{\infty} )(y)| &= \bigg|\dint_{\mathbb{G}^{m}} \dfrac{f_{1}^{\infty}(z_{1}) \cdots f_{m}^{\infty}(z_{m})}{(\rho(z_{1}^{-1}y)+\cdots+\rho(z_{m}^{-1}y))^{mQ-\alpha}}  \mathd  \vec{z} \bigg| \\
  &\le  C  \dint_{(\mathbb{G}\setminus B(x,2r))^{m}} \dfrac{|f_{1}(z_{1}) \cdots f_{m}(z_{m})|}{|\rho(z_{1}^{-1}y)+\cdots+\rho(z_{m}^{-1}y)|^{mQ-\alpha}}  \mathd  \vec{z}  \\
   &\le C \prod_{i=1}^{m}  \dint_{\mathbb{G}\setminus B(x,2r)} \dfrac{|f_{i}(z_{i})|}{\rho(z_{i}^{-1}y)^{Q-\alpha_{i}}}  \mathd z_{i}   \\
   &\le C \prod_{i=1}^{m}  \dint_{\mathbb{G}\setminus B(x,2r)} \dfrac{|f_{i}(z_{i})|}{\rho(z_{i}^{-1}x)^{Q-\alpha_{i}}}  \mathd z_{i}.
%----------
\end{align*}
%------------
Further, by applying Fubini's theorem, H\"{o}lder's inequality (see \cref{lem:holder-inequality-Lie-group}) and \cref{lem:norm-characteristic-functions-Lie-group}, we have that
%----------------
\begin{align*}
%-------
 \dint_{\mathbb{G}\setminus B(x,2r)} \dfrac{|f_{i}(z_{i})|}{\rho(z_{i}^{-1}x)^{Q-\alpha_{i}}}  \mathd z_{i} &= \dint_{\mathbb{G}\setminus B(x,2r)} |f_{i}(z_{i})| \rho(z_{i}^{-1}x)^{\alpha_{i}-Q}  \mathd z_{i} \\
 &\le C\dint_{\mathbb{G}\setminus B(x,2r)} |f_{i}(z_{i})| \bigg(\dint_{\rho(z_{i}^{-1}x)}^{\infty} t^{\alpha_{i}-Q-1} \mathd t\bigg) \mathd z_{i} \\
 &\approx C \dint_{2r}^{\infty} \bigg( \dint_{ 2r<\rho(z_{i}^{-1}x)<t} |f_{i}(z_{i})| \mathd z_{i}\bigg)   t^{\alpha_{i}-Q-1} \mathd t  \\
  &\le  C \dint_{2r}^{\infty} \bigg( \dint_{B(x,t)} |f_{i}(z_{i})| \mathd z_{i}\bigg)   t^{\alpha_{i}-Q-1} \mathd t  \\
  &\le  C \dint_{2r}^{\infty}    t^{\alpha_{i}-Q-1} \|f_{i}\|_{L^{p_{i}}(B(x,t))}  \|\dchi_{B(x,t)}\|_{L^{p'_{i}}(B(x,t))}  \mathd t  \\
  &\le  C \dint_{2r}^{\infty}    t^{\alpha_{i}-\frac{Q}{p_{i}}-1} \|f_{i}\|_{L^{p_{i}}(B(x,t))}   \mathd t.
%----------
\end{align*}
%------------

 Consequently,  taking into account the obtained estimate above, we conclude that
%----------------
\begin{align*}
%-------
  E_{2\infty} &= \| \mathcal{I}_{\alpha,m} ( f_{1}^{\infty} ,\ldots, f_{m}^{\infty} ) \|_{L^{q}(B(x,r))} =  \bigg( \dint_{B(x,r)} | \mathcal{I}_{\alpha,m} ( f_{1}^{\infty} ,\ldots, f_{m}^{\infty} )(y)|^{q}   \mathd y \bigg)^{1/q}  \\
  &\le  C  r^{\frac{Q}{q}} \prod_{i=1}^{m}   \int_{2r}^{\infty} t^{\alpha_{i}-\frac{Q}{p_{i}}-1} \|f_{i}\|_{L^{p_{i}}(B(x,t))} \mathd t.
%----------
\end{align*}
%------------

 Now, for $(\beta_{1},\dots,\beta_{m})\in \ell$, let us consider the terms $ E_{2(\beta_{1},\dots,\beta_{m})}$ such that at least one
 $\beta_{i}=0$ and  one  $\beta_{j}=\infty$. Without loss of generality, we assume that  $ \beta_{1}=\cdots=\beta_{k}=0$ and $\beta_{k+1}=\cdots=\beta_{m}=\infty$ with $1\le k<m$. It is easy to  check  that $ \rho(z_{i}^{-1}x) \approx\rho(z_{i}^{-1}y)$ since $y\in B(x,r)$ and $z_{i}\in  \sideset{^{\complement}}{}{\mathop {B(x,2r)}}=\mathbb{G}\setminus B(x,2r)~(i=1,2,\ldots,m)$. Thus,  we have
%----------------
\begin{align*}
%-------
 &| \mathcal{I}_{\alpha,m} ( f_{1}^{0},\ldots,f_{k}^{0}, f_{k+1}^{\infty},\ldots, f_{m}^{\infty} )(y)| = \bigg|\dint_{\mathbb{G}^{m}} \dfrac{f_{1}^{0}(z_{1}) \cdots f_{k}^{0}(z_{k})   f_{k+1}^{\infty}(z_{k+1}) \cdots f_{m}^{\infty}(z_{m})}{(\rho(z_{1}^{-1}y)+\cdots+\rho(z_{m}^{-1}y))^{mQ-\alpha}}  \mathd  \vec{z} \bigg| \\
  &\qquad \le  C  \dint_{\mathbb{G}^{m}} \dfrac{|f_{1}^{0}(z_{1}) \cdots f_{k}^{0}(z_{k})   f_{k+1}^{\infty}(z_{k+1}) \cdots f_{m}^{\infty}(z_{m})|}{|\rho(z_{1}^{-1}y)+\cdots+\rho(z_{m}^{-1}y)|^{mQ-\alpha}}  \mathd  \vec{z}  \\
   &\qquad \le C\bigg( \prod_{i=1}^{k} \dint_{B(x,2r)} \dfrac{|f_{i}(z_{i})|}{\rho(z_{i}^{-1}y)^{Q-\alpha_{i}}}  \mathd z_{i}  \bigg)\bigg( \prod_{j=k+1}^{m} \dint_{\mathbb{G}\setminus B(x,2r)} \dfrac{|f_{i}(z_{j})|}{\rho(z_{j}^{-1}y)^{Q-\alpha_{j}}}  \mathd z_{j}  \bigg)  \\
   &\qquad \le C\bigg( \prod_{i=1}^{k} \dint_{B(x,2r)} \dfrac{|f_{i}(z_{i})|}{\rho(z_{i}^{-1}x)^{Q-\alpha_{i}}}  \mathd z_{i}  \bigg)\bigg( \prod_{j=k+1}^{m} \dint_{\mathbb{G}\setminus B(x,2r)} \dfrac{|f_{i}(z_{j})|}{\rho(z_{j}^{-1}x)^{Q-\alpha_{j}}}  \mathd z_{j}  \bigg).
%----------
\end{align*}
%------------

 Similar to the estimate above, using Fubini's theorem, H\"{o}lder's inequality (see \cref{lem:holder-inequality-Lie-group}), \cref{lem:norm-characteristic-functions-Lie-group} and \labelcref{inequ:norm-single-integral}, we get
%----------------
\begin{align*}
%-------
  E_{2(\beta_{1},\dots,\beta_{m})} &= \| \mathcal{I}_{\alpha,m} ( f_{1}^{0},\ldots,f_{k}^{0}, f_{k+1}^{\infty},\ldots, f_{m}^{\infty} ) \|_{L^{q}(B(x,r))} \\
  & =  \bigg( \dint_{B(x,r)} | \mathcal{I}_{\alpha,m} ( f_{1}^{0},\ldots,f_{k}^{0}, f_{k+1}^{\infty},\ldots, f_{m}^{\infty} )(y)|^{q}   \mathd y \bigg)^{1/q}  \\
%   &\le C\bigg( \prod_{i=1}^{k} \dint_{B(x,2r)} \dfrac{|f_{i}(z_{i})|}{\rho(z_{i}^{-1}x)^{Q-\alpha_{i}}}  \mathd z_{i}  \bigg)\bigg( \prod_{j=k+1}^{m} \dint_{\mathbb{G}\setminus B(x,2r)} \dfrac{|f_{i}(z_{j})|}{\rho(z_{j}^{-1}x)^{Q-\alpha_{j}}}  \mathd z_{j}  \bigg) \|\dchi_{B(x,r)}\|_{L^{q}(\mathbb{G})}  \\
   &\le C r^{\frac{Q}{q}}\bigg( \prod_{i=1}^{k} \dint_{B(x,2r)} \dfrac{|f_{i}(z_{i})|}{\rho(z_{i}^{-1}x)^{Q-\alpha_{i}}}  \mathd z_{i}  \bigg)\bigg( \prod_{j=k+1}^{m} \dint_{\mathbb{G}\setminus B(x,2r)} \dfrac{|f_{i}(z_{j})|}{\rho(z_{j}^{-1}x)^{Q-\alpha_{j}}}  \mathd z_{j}  \bigg)  \\
  &\le  C  r^{\frac{Q}{q}} \prod_{i=1}^{m}   \int_{2r}^{\infty} t^{\alpha_{i}-\frac{Q}{p_{i}}-1} \|f_{i}\|_{L^{p_{i}}(B(x,t))} \mathd t.
%----------
\end{align*}
%------------

Combining the above estimates we get the desired result. The proof is completed.

%--------------
\end{proof}
%------------

%--------------------
\begin{refproof}[Proof of \cref{thm: frac-int-op-main-1}]
%---------------
Let  $1< p_{i}<\infty~(i=1,2,\ldots,m)$ and  $\vec{f}=(f_{1},f_{2},\dots,f_{m}) \in \mathcal{L}^{p_{1},\varphi_{1}}(\mathbb{G})\times \cdots \times  \mathcal{L}^{p_{m},\varphi_{m}}(\mathbb{G})$. According to the assumption \labelcref{condition:Zygmund-type-integral-inequality}, and using \cref{lem:frac-loc-lebesgue} and \cref{lem:thm-3.1-guliyev2013generalized}  with $w(r)=r^{-Q/q_{i}-1}$, $v_{2}(r)= \psi(x,r)^{-1/m}$ and $v_{1}(r)= \varphi_{i}(x,r)^{-1} r^{-Q/p_{i}} ~(i=1,2,\ldots,m)$, we have
%----------------
\begin{align*}
%-------
  \|\mathcal{I}_{\alpha,m}(\vec{f})\|_{\mathcal{L}^{q,\psi}(\mathbb{G})}  &= \sup_{x \in \mathbb{G} \atop r>0 } \dfrac{1}{\psi(x,r)} \Big( \frac{1}{\vert B(x,r)\vert } \dint_{B(x,r)} \vert \mathcal{I}_{\alpha,m}(\vec{f})(y)\vert ^{q} \mathd y \Big)^{1/q}   \\
%   &\le C\sup_{x \in \mathbb{G} \atop r>0 }  \psi(x,r)^{-1} \prod_{i=1}^{m} \dint_{2r}^{\infty} t^{\alpha_{i}-\frac{Q}{p_{i}}-1} \|f_{i}\|_{L^{p_{i}}(B(x,t))}\mathd t   \\
   &\le C\sup_{x \in \mathbb{G} \atop r>0 }   \prod_{i=1}^{m} \psi(x,r)^{-1/m}\dint_{r}^{\infty} t^{-\frac{Q}{q_{i}}-1} \|f_{i}\|_{L^{p_{i}}(B(x,t))}\mathd t   \\
   &\le C\sup_{x \in \mathbb{G} \atop r>0 }   \prod_{i=1}^{m} \varphi_{i}(x,r)^{-1} r^{-Q/p_{i}}  \|f_{i}\|_{L^{p_{i}}(B(x,r))}  \\
  &\le C   \prod_{i=1}^{m} \|f_{i}\|_{\mathcal{L}^{p_{i},\varphi_{i}}(\mathbb{G})}
%----------
\end{align*}
%------------

 This completes the proof of \cref{thm: frac-int-op-main-1}.
%--------------------------------
\end{refproof}
%-------------------

\subsection{Proof of \cref{thm: frac-max-op-main-2}}

In order to prove \cref{thm: frac-max-op-main-2}, we also need the follow auxiliary results.

Similar to the  pointwise relation between fraction  integral operator and fractional maximal operator, by elementary calculations, we can obtain the  following   lemma, and   omit the proof.
%----------------
\begin{lemma}    \label{lem:multi-pointwise-relation}
%---------------
Suppose that  $m\in \mathbb{Z}^{+}$,  $0<\alpha_{i} <Q $,  $1< p_{i}<Q/\alpha_{i}~(i=1,2,\ldots,m)$ and $\alpha=\sum\limits_{i=1}^{m}\alpha_{i}$. Let $ \vec{f} \in L^{p_{1}}(\mathbb{G})\times \cdots \times  L^{p_{m}}(\mathbb{G})$, then there exists a positive
constant $C$ such that  the pointwise inequality
%----------------
\begin{align*}
%-------
 \mathcal{M}_{\alpha,m}(\vec{f})(x) \le C \mathcal{I}_{\alpha,m}(|f_{1}|,\ldots,|f_{m}|)(x)
%----------
\end{align*}
%------------
  holds for any $x\in\mathbb{G}$.  % where  $ C>0$ independing on $r>0$ and $x\in \mathbb{G}$.
%----------
\end{lemma}
%------------

According to \cref{lem:thm-3.1-liu2019multilinear} and \cref{lem:multi-pointwise-relation},  the following result  can be  obtained, and we omit the proof.
%----------------
\begin{lemma}    \label{lem:frac-max}
%---------------
Suppose that  $m\in \mathbb{Z}^{+}$,  $0<\alpha_{i} <Q $,  $1< p_{i}<Q/\alpha_{i}~(i=1,2,\ldots,m)$ and $\alpha=\sum\limits_{i=1}^{m}\alpha_{i}$. Let  $q$ satisfy
%----------------
\begin{align*}
%-------
  \frac{1}{q}=\frac{1}{p_{1}}+\cdots+\frac{1}{p_{m}}-\frac{\alpha}{Q}<1.
%----------
\end{align*}
%------------
Then the operator $\mathcal{M}_{\alpha,m}$ is bounded from product space  $ L^{p_{1}}(\mathbb{G})\times \cdots \times  L^{p_{m}}(\mathbb{G})$ to $ L^{q}(\mathbb{G})$, namely
%----------------
\begin{align*}
%-------
  \|\mathcal{M}_{\alpha,m}(\vec{f})\|_{L^{q}(\mathbb{G})}\le C\prod_{i=1}^{m} \|f_{i}\|_{L^{p_{i}}(\mathbb{G})}.
%----------
\end{align*}
%------------
%----------
\end{lemma}
%------------

In addition, the following local estimates are true.  %is also needed.
Specifically, when m=1, the result can be found in \cite{guliyev2013boundedness} (see Lemma 3.2).
%----------------
\begin{lemma}    \label{lem:frac-max-loc-lebesgue}
%---------------
Suppose that  $m\in \mathbb{Z}^{+}$,  $0<\alpha_{i} <Q $,  $1< p_{i}<Q/\alpha_{i}~(i=1,2,\ldots,m)$ and $\alpha=\sum\limits_{i=1}^{m}\alpha_{i}$. Let  $q$ satisfy
%----------------
\begin{align*}
%-------
  \frac{1}{q}=\frac{1}{p_{1}}+\cdots+\frac{1}{p_{m}}-\frac{\alpha}{Q}<1.
%----------
\end{align*}
%------------
Then  the inequality
%----------------
\begin{align*}
%-------
  \|\mathcal{M}_{\alpha,m}(\vec{f})\|_{L^{q}(B(x,r))}\le C r^{Q/q}\prod_{i=1}^{m} \sup_{t>2r} t^{\alpha_{i}-\frac{Q}{p_{i}}} \|f_{i}\|_{L^{p_{i}}(B(x,t))}
%----------
\end{align*}
%------------
holds for any ball $B(x,r)$ and for all  $\vec{f}\in L_{\loc}^{p_{1}}(\mathbb{G})\times \cdots \times  L_{\loc}^{p_{m}}(\mathbb{G})$.
%----------
\end{lemma}
%------------

%--------------------
\begin{proof}
%---------------
Similar to the proof of \cref{lem:frac-loc-lebesgue},
%For arbitrary $x\in\mathbb{G}$, set $B=B(x,r)$ be the $\mathbb{G}$ ball centered at $x$ with radius $r$, and $2B=B(x,2r)$. For each $j$, we decompose $f_{j}=f_{j}^{0}+f_{j}^{\infty}$ with $f_{j}^{0}=f_{j}\dchi_{2B}$. Then
for each $j$, we decompose $f_{j}=f_{j}^{0}+f_{j}^{\infty}$ with $f_{j}^{0}=f_{j}\dchi_{2B}$, and
\begin{align*}
 \prod_{j=1}^{m} f_{j}   %&=  \prod_{j=1}^{m} \Big( f_{j}^{0} + f_{j}^{\infty} \Big)     =   \sum_{\beta_{1},\dots,\beta_{m}\in \{0,\infty\}}  f_{1}^{\beta_{1}} \cdots  f_{m}^{\beta_{m}}  \\
  &=  \prod_{j=1}^{m} f_{j}^{0}  +  \sum_{(\beta_{1},\dots,\beta_{m})\in \ell} f_{1}^{\beta_{1}}  \cdots  f_{m}^{\beta_{m}},
\end{align*}
where $\ell=\{ (\beta_{1},\dots,\beta_{m}): \hbox{there is at least one}~ \beta_{j}\neq 0\}$.
Thus, for arbitrary $y\in B(x,r)$, we  obtain
\begin{align*}
\mathcal{M}_{\alpha,m}(\vec{f})(y) = \mathcal{M}_{\alpha,m} ( f_{1}^{0} ,\ldots, f_{m}^{0} )(y) +  \sum_{(\beta_{1},\dots,\beta_{m})\in \ell} \mathcal{M}_{\alpha,m} ( f_{1}^{\beta_{1}} ,\ldots, f_{m}^{\beta_{m}} ) (y).
\end{align*}
Then,
\begin{align*}
\|\mathcal{M}_{\alpha,m}(\vec{f})\|_{L^{q}(B(x,r))} &\leq \|\mathcal{M}_{\alpha,m} ( f_{1}^{0} ,\ldots, f_{m}^{0} )\|_{L^{q}(B(x,r))}
  + \bigg\|\sum_{(\beta_{1},\dots,\beta_{m})\in \ell} \mathcal{M}_{\alpha,m} ( f_{1}^{\beta_{1}} ,\ldots, f_{m}^{\beta_{m}} ) \bigg\|_{L^{q}(B(x,r))}   \\
&= E_{1}+E_{2}.
\end{align*}

% For any $z_{i}\in  B(x,2r)~(i=1,2,\ldots,m)$,  applying the     doubling condition of Haar measure,   we obtain
%%----------------
%\begin{align*}
%%-------
% \|f_{i}\|_{L^{p_{i}}(B(x,2r))} &=\bigg(\dint_{B(x,2r)} |f_{i}(z_{i})|^{p_{i}}   \mathd z_{i}\bigg)^{1/p_{i}} \\
% &= \frac{|B(x,2r)|^{1/q_{i}}}{|B(x,2r)|^{1/q_{i}}}\bigg(\dint_{B(x,2r)} |f_{i}(z_{i})|^{p_{i}}   \mathd z_{i}\bigg)^{1/p_{i}}  \\
% &\le C r^{\frac{Q}{q_{i}}} \frac{1}{|B(x,2r)|^{1/q_{i}}}\bigg(\dint_{B(x,2r)} |f_{i}(z_{i})|^{p_{i}}   \mathd z_{i}\bigg)^{1/p_{i}}   \\
% &\le  C r^{\frac{Q}{q_{i}}} \sup_{t>2r}\frac{1}{|B(x,t)|^{1/q_{i}}}\bigg(\dint_{B(x,t)} |f_{i}(z_{i})|^{p_{i}}   \mathd z_{i}\bigg)^{1/p_{i}}     \\
%  &\le  C r^{\frac{Q}{q_{i}}}  \sup_{t>2r} t^{\alpha_{i}-\frac{Q}{p_{i}}} \|f_{i}\|_{L^{p_{i}}(B(x,t))}.
%%----------
%\end{align*}
%%------------

For $E_{1}$, applying  the boundedness of  $\mathcal{M}_{\alpha,m}$ (see \cref{lem:frac-max}) and the     doubling condition of Haar measure, we have
%----------------
\begin{align*}
%-------
  E_{1} &= \|\mathcal{M}_{\alpha,m} ( f_{1}^{0} ,\ldots, f_{m}^{0} )\|_{L^{q}(B(x,r))} \le \|\mathcal{M}_{\alpha,m}( f_{1}^{0} ,\ldots, f_{m}^{0} )\|_{L^{q}(\mathbb{G})} \\
  &\le C\prod_{i=1}^{m} \|f_{i}^{0}\|_{L^{p_{i}}(\mathbb{G})}\le C\prod_{i=1}^{m} \|f_{i}\|_{L^{p_{i}}(B(x,2r))}  \\
%   &\le  C\prod_{i=1}^{m}   r^{\frac{Q}{q_{i}}}  \sup_{t>2r} t^{\alpha_{i}-\frac{Q}{p_{i}}} \|f_{i}\|_{L^{p_{i}}(B(x,t))}   \\
  &\le  C  r^{\frac{Q}{q}} \prod_{i=1}^{m}   \sup_{t>2r} t^{\alpha_{i}-\frac{Q}{p_{i}}} \|f_{i}\|_{L^{p_{i}}(B(x,t))} .
%----------
\end{align*}
%------------

To estimate $E_{2}$, we first consider the case   $\beta_{1}=\beta_{2}=\cdot\cdot\cdot=\beta_{m}=\infty$.
%When $y\in B(x,r)$ and $z_{i}\in  \sideset{^{\complement}}{}{\mathop {B(x,2r)}}=\mathbb{G}\setminus B(x,2r)~(i=1,2,\ldots,m)$,   the conditions imply $\frac{1}{2}\rho(z_{i}^{-1}y) \leq \rho(z_{i}^{-1}x) \leq \frac{3}{2}\rho(z_{i}^{-1}y)$. Thus, we can obtain
Let $y$ be an arbitrary point from $B(x,r)$. If  $B(y,t) \bigcap \sideset{^{\complement}}{}{\mathop {B(x,2r)}}\neq \emptyset$, then $t>r$. In fact, when $z_{i}\in  B(y,t) \bigcap \sideset{^{\complement}}{}{\mathop {B(x,2r)}}~(i=1,2,\ldots,m)$, we have
%----------------
\begin{align*}
%-------
 t> \rho(z_{i}^{-1}y) \ge \rho(z_{i}^{-1}x)- \rho(y^{-1}x)>2r-r=r.
%----------
\end{align*}
%------------
On the other hand,  $B(y,t) \bigcap \sideset{^{\complement}}{}{\mathop {B(x,2r)}} \subset B(x,2t)  $. Indeed, when $z_{i}\in  B(y,t) \bigcap \sideset{^{\complement}}{}{\mathop {B(x,2r)}}~(i=1,2,\ldots,m)$, we get
%----------------
\begin{align*}
%-------
   \rho(z_{i}^{-1}x)\le  \rho(z_{i}^{-1}y)  + \rho(y^{-1}x)<t+r<2t.
%----------
\end{align*}
%------------
Then, for all $y\in B(x,r)$ and any $z_{i}\in  B(y,t) \bigcap \sideset{^{\complement}}{}{\mathop {B(x,2r)}}~(i=1,2,\ldots,m)$, using H\"{o}lder's inequality (see \cref{lem:holder-inequality-Lie-group})  and \cref{lem:norm-characteristic-functions-Lie-group},   we have
%----------------
\begin{align*}
%-------
  \mathcal{M}_{\alpha,m} ( f_{1}^{\infty} ,\ldots, f_{m}^{\infty} )(y) &= \sup_{B(x,r)\ni y \atop t>0}|B(y,t)|^{\frac{\alpha}{Q}} \prod_{i=1}^{m} \dfrac{1}{|B(y,t)|} \dint_{B(y,t)}  |f_{i}^{\infty}(z_{i}) |    \mathd  z_{i}  \\
  &= \sup_{B(x,r)\ni y \atop t>0}|B(y,t)|^{\frac{\alpha}{Q}} \prod_{i=1}^{m} \dfrac{1}{|B(y,t)|} \dint_{B(y,t)\bigcap \sideset{^{\complement}}{}{\mathop {B(x,2r)}}}  |f_{i}(z_{i}) |    \mathd  z_{i}  \\
  &\le  C  \sup_{ t>r}|B(x,2t)|^{\frac{\alpha}{Q}} \prod_{i=1}^{m} \dfrac{1}{|B(x,2t)|} \dint_{B(x,2t) }  |f_{i}(z_{i}) |    \mathd  z_{i}  \\
   &\le C  \sup_{  t>2r}|B(x,t)|^{\frac{\alpha}{Q}} \prod_{i=1}^{m} \dfrac{1}{|B(x,t)|} \dint_{B(x,t) }  |f_{i}(z_{i}) |    \mathd  z_{i}    \\
   &\le C  \sup_{  t>2r}|B(x,t)|^{\frac{\alpha}{Q}} \prod_{i=1}^{m} \dfrac{1}{|B(x,t)|}  \|f_{i}\|_{L^{p_{i}}(B(x,t))} \|\dchi_{B(x,t)}\|_{L^{p'_{i}}(B(x,t))}    \\
   &\le C  \sup_{  t>2r}|B(x,t)|^{\frac{\alpha}{Q}} \prod_{i=1}^{m}  |B(x,t)|^{-1/p_{i}}   \|f_{i}\|_{L^{p_{i}}(B(x,t))}      \\
   &\le C \prod_{i=1}^{m} \sup_{  t>2r}     t^{\alpha_{i}-\frac{Q}{p_{i}}} \|f_{i}\|_{L^{p_{i}}(B(x,t))} .
%   &\le C  \sup_{  t>2r} \prod_{i=1}^{m}    t^{\alpha_{i}-\frac{Q}{p_{i}}} \|f_{i}\|_{L^{p_{i}}(B(x,t))} .
%----------
\end{align*}
%------------

 Therefore,   we conclude that
%----------------
\begin{align*}
%-------
  E_{2\infty} &= \| \mathcal{M}_{\alpha,m} ( f_{1}^{\infty} ,\ldots, f_{m}^{\infty} ) \|_{L^{q}(B(x,r))} =  \bigg( \dint_{B(x,r)} | \mathcal{M}_{\alpha,m} ( f_{1}^{\infty} ,\ldots, f_{m}^{\infty} )(y)|^{q}   \mathd y \bigg)^{1/q}  \\
  &\le  C  r^{\frac{Q}{q}} \prod_{i=1}^{m} \sup_{  t>2r} t^{\alpha_{i}-\frac{Q}{p_{i}}} \|f_{i}\|_{L^{p_{i}}(B(x,t))}.
%----------
\end{align*}
%------------

 Now, for $(\beta_{1},\dots,\beta_{m})\in \ell$, let us consider the terms $ E_{2(\beta_{1},\dots,\beta_{m})}$ such that at least one
 $\beta_{i}=0$ and  one  $\beta_{j}=\infty$. Without loss of generality, we assume that  $ \beta_{1}=\cdots=\beta_{k}=0$ and $\beta_{k+1}=\cdots=\beta_{m}=\infty$ with $1\le k<m$.
%Then, using \cref{lem:frac-max}, the     doubling condition of Haar measure, H\"{o}lder's inequality (see \cref{lem:holder-inequality-Lie-group})  and \cref{lem:norm-characteristic-functions-Lie-group},
Then, for all $y\in B(x,r)$, using H\"{o}lder's inequality (see \cref{lem:holder-inequality-Lie-group})  and \cref{lem:norm-characteristic-functions-Lie-group},  we obtain that
%----------------
\begin{align*}
%-------
  \mathcal{M}_{\alpha,m} ( f_{1}^{0},\ldots,f_{k}^{0}, f_{k+1}^{\infty},\ldots, f_{m}^{\infty} )(y)
 &= \mathcal{M}_{\alpha,k} ( f_{1}^{0},\ldots,f_{k}^{0} )(y)  \mathcal{M}_{\alpha,m-k} (  f_{k+1}^{\infty},\ldots, f_{m}^{\infty} )(y)   \\
 &\le C \bigg( \prod_{i=1}^{k} M_{\alpha} ( f_{i}^{0})(y)  \bigg)\bigg( \prod_{j=k+1}^{m} \sup_{  t>2r}     t^{\alpha_{j}-\frac{Q}{p_{j}}} \|f_{j}\|_{L^{p_{j}}(B(x,t))} \bigg).
%----------
\end{align*}
%------------

 Similar to the estimates $E_{1}$ and $E_{2\infty}$,   using \cref{lem:frac-max} with $m=1$, the     doubling condition of Haar measure, H\"{o}lder's inequality (see \cref{lem:holder-inequality-Lie-group})  and \cref{lem:norm-characteristic-functions-Lie-group}, we get
%----------------
\begin{align*}
%-------
  E_{2(\beta_{1},\dots,\beta_{m})} &= \| \mathcal{M}_{\alpha,m} ( f_{1}^{0},\ldots,f_{k}^{0}, f_{k+1}^{\infty},\ldots, f_{m}^{\infty} ) \|_{L^{q}(B(x,r))} \\
  &=  \bigg( \dint_{B(x,r)} | \mathcal{M}_{\alpha,m} ( f_{1}^{0},\ldots,f_{k}^{0}, f_{k+1}^{\infty},\ldots, f_{m}^{\infty} )(y)|^{q}   \mathd y \bigg)^{1/q}  \\
%   &\le C\bigg( \prod_{i=1}^{k} \|M_{\alpha} ( f_{i}^{0})\|_{L^{q_{i}}(B(x,r))}    \bigg)\bigg( \prod_{j=k+1}^{m} \sup_{  t>2r}     t^{\alpha_{j}-\frac{Q}{p_{j}}} \|f_{j}\|_{L^{p_{j}}(B(x,t))} \|\dchi_{B(x,r)}\|_{L^{q_{j}}(\mathbb{G})}   \bigg)   \\
   &\le C r^{\frac{Q}{q}}\bigg( \prod_{i=1}^{k} \sup_{t>2r} t^{\alpha_{i}-\frac{Q}{p_{i}}} \|f_{i}\|_{L^{p_{i}}(B(x,t))}   \bigg)\bigg( \prod_{j=k+1}^{m} \sup_{  t>2r}     t^{\alpha_{j}-\frac{Q}{p_{j}}} \|f_{j}\|_{L^{p_{j}}(B(x,t))}   \bigg)  \\
  &\le  C  r^{\frac{Q}{q}} \prod_{i=1}^{m}  \sup_{t>2r} t^{\alpha_{i}-\frac{Q}{p_{i}}} \|f_{i}\|_{L^{p_{i}}(B(x,t))} .
%----------
\end{align*}
%------------

Combining the above estimates we get the desired result. The proof is completed.

%--------------
\end{proof}
%------------

The following supremal type inequality plays a key role in the proof of \cref{thm: frac-max-op-main-2}, which can be founded in \cite{burenkov2010boundedness}(see Theorem 5.4 or Theorem 3.1 in \cite{guliyev2013boundedness}).
%----------------
\begin{lemma}    \label{lem:supremal-type-inequality-thm3.1-guliyev2013boundedness}
%---------------
Let $v_{1}$ and $v_{2}$   be non-negative measurable functions satisfying $ 0<\|v_{1} \|_{L^{\infty}(t,\infty)}<\infty$, $ 0<\|v_{2} \|_{L^{\infty}(0,t)}<\infty$  for any $t \in (0,\infty)$. And let $u$ be a continuous non-negative function on $(0,\infty)$. Then the supremal operator $\overline{S}_{u}$ is bounded from $L_{v_{1}}^{\infty}(0,\infty)$ to $L_{v_{2}}^{\infty}(0,\infty)$ on the cone $\mathbb{A}$ if and only if
%----------------
\begin{align*}
%-------
  \left\|v_{2} \overline{S}_{u} \Big(\|v_{1} \|_{L^{\infty}(\cdot,\infty)}^{-1} \Big) \right\|_{L^{\infty}(0,\infty)} <\infty.
%----------
\end{align*}
%------------
%----------
\end{lemma}
%------------

%--------------------
\begin{refproof}[Proof of \cref{thm: frac-max-op-main-2}]
%---------------
Let $u(r)=r^{-Q/q_{i}}$, $v_{2}(r)= \psi(x,r)^{-1/m}$ and $v_{1}(r)= \varphi_{i}(x,r)^{-1} r^{-Q/p_{i}} ~(i=1,2,\ldots,m)$,  According to the   hypothesis   \labelcref{condition:supremal-type-integral-inequality}, it follows that
%----------------
\begin{align*}
%-------
  \left\|v_{2} \overline{S}_{u} \Big(\|v_{1} \|_{L^{\infty}(\cdot,\infty)}^{-1} \Big) \right\|_{L^{\infty}(0,\infty)}
%  &= \esssup_{r>0} v_{2}(r) \overline{S}_{u} \Big(\|v_{1} \|_{L^{\infty}(\cdot,\infty)}^{-1} \Big)(r)   \\
%   &= \esssup_{r>0} v_{2}(r) \left\|  u  \|v_{1} \|_{L^{\infty}(\cdot,\infty)}^{-1}   \right\|_{L^{\infty}(r,\infty)}  \\
%   &= \esssup_{r>0} v_{2}(r) \left(\esssup_{t>r}  u(t)  \|v_{1} \|_{L^{\infty}(t,\infty)}^{-1}   \right)  \\
%   &= \esssup_{r>0} v_{2}(r) \left(\esssup_{ r<t<\infty}  u(t)  \left(\essinf_{t<s<\infty}  v_{1}^{-1}(s)  \right)  \right)  \\
   &\le C.
%----------
\end{align*}
%------------

Set  $1< p_{i}<\infty~(i=1,2,\ldots,m)$ and  $\vec{f}=(f_{1},f_{2},\dots,f_{m}) \in \mathcal{L}^{p_{1},\varphi_{1}}(\mathbb{G})\times \cdots \times  \mathcal{L}^{p_{m},\varphi_{m}}(\mathbb{G})$.
% Using \cref{lem:frac-max-loc-lebesgue} and \cref{lem:supremal-type-inequality-thm3.1-guliyev2013boundedness} with $u(r)=r^{-Q/q_{i}}$, $v_{2}(r)= \psi(x,r)^{-1/m}$ and $v_{1}(r)= \varphi_{i}(x,r)^{-1} r^{-Q/p_{i}} ~(i=1,2,\ldots,m)$, we obtain
%
%Let $u(r)=r^{-Q/q_{i}}$, $v_{2}(r)= \psi(x,r)^{-1/m}$ and $v_{1}(r)= \varphi_{i}(x,r)^{-1} r^{-Q/p_{i}} ~(i=1,2,\ldots,m)$, then we have
%%----------------
%\begin{align*}
%%-------
%  \left\|v_{2} \overline{S}_{u} \Big(\|v_{1} \|_{L^{\infty}(\cdot,\infty)}^{-1} \Big) \right\|_{L^{\infty}(0,\infty)} &= \esssup_{r>0} v_{2}(r) \overline{S}_{u} \Big(\|v_{1} \|_{L^{\infty}(\cdot,\infty)}^{-1} \Big)(r)   \\
%   &= \esssup_{r>0} v_{2}(r) \left\|  u  \|v_{1} \|_{L^{\infty}(\cdot,\infty)}^{-1}   \right\|_{L^{\infty}(r,\infty)}  \\
%   &= \esssup_{r>0} v_{2}(r) \left(\esssup_{t>r}  u(t)  \|v_{1} \|_{L^{\infty}(t,\infty)}^{-1}   \right)  \\
%   &= \esssup_{r>0} v_{2}(r) \left(\esssup_{ r<t<\infty}  u(t)  \left(\essinf_{t<s<\infty}  v_{1}^{-1}(s)  \right)  \right)  \\
%   &\le C.
%%----------
%\end{align*}
%%------------
  Using \cref{lem:frac-max-loc-lebesgue} and \cref{lem:supremal-type-inequality-thm3.1-guliyev2013boundedness}, we obtain
%----------------
\begin{align*}
%-------
  \|\mathcal{M}_{\alpha,m}(\vec{f})\|_{\mathcal{L}^{q,\psi}(\mathbb{G})}  &= \sup_{x \in \mathbb{G} \atop r>0 } \dfrac{1}{\psi(x,r)} \Big( \frac{1}{\vert B(x,r)\vert } \dint_{B(x,r)} \vert \mathcal{M}_{\alpha,m}(\vec{f})(y)\vert ^{q} \mathd y \Big)^{1/q}   \\
%   &=\sup_{x \in \mathbb{G} \atop r>0 }  \psi(x,r)^{-1}  r^{-Q/q}  \|\mathcal{M}_{\alpha,m}(\vec{f})\|_{L^{q}(B(x,r))}   \\
%   &\le C\sup_{x \in \mathbb{G} \atop r>0 }  \psi(x,r)^{-1}  r^{-Q/q}  r^{Q/q}\prod_{i=1}^{m} \sup_{t>2r} t^{\alpha_{i}-\frac{Q}{p_{i}}} \|f_{i}\|_{L^{p_{i}}(B(x,t))}   \\
%   &\le C\sup_{x \in \mathbb{G} \atop r>0 }  \psi(x,r)^{-1}  \prod_{i=1}^{m} \sup_{t>2r} t^{\alpha_{i}-\frac{Q}{p_{i}}} \|f_{i}\|_{L^{p_{i}}(B(x,t))}   \\
   &\le C  \prod_{i=1}^{m} \sup_{x \in \mathbb{G} \atop r>0 } \psi(x,r)^{-1/m} \sup_{t>2r} t^{\alpha_{i}-\frac{Q}{p_{i}}} \|f_{i}\|_{L^{p_{i}}(B(x,t))}   \\
   &\le C   \prod_{i=1}^{m} \sup_{x \in \mathbb{G} \atop r>0 } \varphi_{i}(x,r)^{-1} r^{-Q/p_{i}}  \|f_{i}\|_{L^{p_{i}}(B(x,r))}  \\
  &\le C   \prod_{i=1}^{m} \|f_{i}\|_{\mathcal{L}^{p_{i},\varphi_{i}}(\mathbb{G})}.
%----------
\end{align*}
%------------

 This completes the proof of \cref{thm: frac-max-op-main-2}.
%--------------------------------

%--------------
\end{refproof}
%------------

%-----------------------
% \subsubsection*{Acknowledgments:}
%The authors cordially  thank the anonymous referees who gave valuable  suggestions and useful comments which have lead to the improvement of this paper.

%-----------------------
 \subsubsection*{Funding information:}
 This work was partly supported by Project of Heilongjiang Province Science and Technology Program (No.2019-KYYWF-0909,1355ZD010), the National Natural Science Foundation of China (Grant No.11571160)  and the  Reform and Development Foundation for Local Colleges and Universities of the Central Government(No.2020YQ07).

%-----------------------
 \subsubsection*{Conflict of interest: }
The authors state that there is no conflict of interest.   %Authors state no conflict of  interest.

%-----------------------
 \subsubsection*{Data availability statement:}
  All data generated or analysed during this study are included in this manuscript.
%-----------------------
 \subsubsection*{Author contributions:}
 All authors contributed equally to the writing of this article.
 All authors read the final manuscript and approved its submission.

%%-----------------------
%\textbf{Disclosure statement:} Authors state no conflict of  interest.
%
%\textbf{Acknowledgments:} The authors cordially  thank the anonymous referees who gave valuable  suggestions and useful comments which have lead to the improvement of this paper.
%
%\textbf{Funding information:} This work was partly supported by the National Natural Science Foundation of China
%(Grant No. 11571160), Scientific Project-HLJ (No.2019-KYYWF-0909,1355ZD010), and the  Reform and Development Foundation for Local Colleges and Universities of the
%Central Government(No.2020YQ07).
%
%%-----------------------

%------------

%============================
%\section*{References}
%---------------------------------
\phantomsection
\addcontentsline{toc}{section}{References}
%----------
\bibliographystyle{tugboat}          %ieeetr 参考文献 编号 是 数字，按字母顺序排列 jIEEEtran 按引用
                               %plain，按字母的顺序排列，比较次序为作者、年度和标题. plainnat 全称  plplain
                               %unsrt，样式同plain，只是按照引用的先后排序.  unsrtnat  achemso
                               %abbrv，类似plain，将月份全拼改为缩写，更显紧凑.
                               %apalike  姓前，名缩写 年代在前 按字母排序 chicago,  munich, asmejour  erae
                               %siam，美国工业和应用数学学会期刊样式.  按字母排序 年在后  共同作者省略  aomalpha  ier  tugboat
                               %acm，美国计算机学会期刊样式.  大写 姓前，名缩写 按字母排序 年在后  ACM-Reference-Format 全称小写 年在前后
                               % elsarticle-num-names 姓后 缩写名
                               % elsarticle-harv 姓前 缩写名  按字母排序   年在前
                               % elsarticle-num 姓后 缩写名 按引用   asmeconf 全称
%----------
\bibliography{wu-reference}

%============================
\end{document}